\documentclass[12pt, reqno]{amsart}
\usepackage{amsmath, amsfonts, amsbsy, amsthm, amscd, graphicx}
\usepackage{amssymb, latexsym, color}
\usepackage{bm}
\usepackage[svgnames,psnames]{xcolor}
\usepackage[colorlinks, citecolor=Green,linkcolor=FireBrick,linktocpage,unicode]{hyperref}

\theoremstyle{definition}
\numberwithin{equation}{section}
\newtheorem{theo}{Theorem}[section]
\newtheorem{defi}[theo]{Definition}

\newtheorem{rema}[theo]{Remark}
\newtheorem{corr}[theo]{Corollary}
\newtheorem{exam}[theo]{Example}
\newtheorem{prop}[theo]{Proposition}

\begin{document}

\title{Deformation of K\"{a}hler Metrics and an Eigenvalue Problem for the Laplacian on a Compact K\"{a}hler Manifold}
\author{Kazumasa Narita}
\thanks{Graduate School of Mathematics, Nagoya University, Furo-cho, 
Chikusa-ku, Nagoya 464-8602, Japan, m19032e@math.nagoya-u.ac.jp}
\date{}

\maketitle

\begin{abstract}
We study an eigenvalue problem for the Laplacian on a compact K\"{a}hler manifold. Considering the $k$-th eigenvalue $\lambda_{k}$ as a functional on the space of K\"{a}hler metrics with fixed volume on a compact complex manifold, we introduce the notion of $\lambda_{k}$-extremal K\"{a}hler metric. We deduce a condition for a K\"{a}hler metric to be $\lambda_{k}$-extremal. As examples, we consider product K\"{a}hler manifolds, compact isotropy irreducible homogeneous K\"{a}hler manifolds and flat complex tori.
\end{abstract}



\section{Introduction}
Let $M$ be a compact manifold of dimension $n$. Given a Riemannian metric $g$ on $M$, the volume $\mbox{Vol}(M,g)$ and the Laplace-Beltrami operator $\Delta_{g}$ are defined. Let $0 < \lambda_{1}(g) \leq \lambda_{2}(g) \leq \cdots \lambda_{k}(g) \leq \cdots$ be the eigenvalues of $\Delta_{g}$. The quantity $\lambda_{k}(g) \mbox{Vol}(M,g)^{2/n}$ is invariant under scaling of the metric $g$. Hersch\cite{Hersch} proved that on a $2$-dimensional sphere $S^{2}$, the scale-invariant quantity $\lambda_{1}(g)\mbox{Area}(g)$ is maximized when $g$ is a round metric. Inspired by the work, Berger\cite{Berger} asked whether 
\begin{equation*}
\Lambda_{1}(M) := \sup_{g} \lambda_{1}(g) \mbox{Vol}(M,g)^{2/n}
\end{equation*}
is finite for a compact manifold $M$ of dimension $n$. For a surface $M$, $\Lambda_{1}(M)$ is bounded by a constant depending on the genus \cite{YY, Karpukhin}. Berger \cite{Berger} also conjectured that for a 2-dimensional torus $T^{2}$, the flat equilateral torus attains $\Lambda_{1}(T^{2})$. Nadirashvili \cite{Nadirashvili} settled Berger's conjecture affirmatively. In the same paper, he proved a theorem that a metric $g$ on a given surface $M$ is extremal for the functional $\lambda_{1}: g \mapsto \lambda_{1}(g)$ with respect to all the volume-preserving deformations of the metric if and only if there exists a finite collection of $\lambda_{1}(g)$-eigenfunctions $\{f_{j}\}_{j=1}^{N}$ such that $F:= (f_{1}, \cdots, f_{N}): (M,g) \rightarrow \mathbf{R}^{N}$ is an isometric minimal immersion into a sphere in $\mathbf{R}^{N}$. Later, El Soufi-Ilias \cite{EI2} simplified the proof of the theorem and generalized it to a compact manifold $M$ of any dimension. 
More explicitly, they proved the following:
 \begin{theo}[\cite{Nadirashvili}, \cite{EI2}]
 \label{Nad-intro}
Let $(M, g)$ be a compact $m$-dimensional Riemannian manifold. The metric $g$ is extremal for the functional $\lambda_{1}$ with respect to all the volume-preserving deformations of the metric if and only if there exists a finite collection of $\lambda_{1}(g)$-eigenfunctions  $\{f_{1}, \ldots, f_{N} \}$ such that $F:= (f_{1}, \cdots, f_{N}): (M,g) \rightarrow \mathbf{R}^{N}$ is an isometric minimal immersion into $S^{N-1}(\sqrt{m/\lambda_{1}(g)}) \subset \mathbf{R}^{N}$. 
\end{theo}
In particular, they showed that the canonical metric on a compact isotropy irreducible homogeneous manifold is extremal for the functional $\lambda_{1}$. For recent remarkable progress in study of $\Lambda_{1}(M)$ for a surface $M$, see \cite{NS}, \cite{Karpukhin2}, \cite{Petrides}, \cite{Ros} and \cite{KV}, for example.

On the other hand, on any manifold $M$ with $n \geq 3$, one can construct a $1$-parameter family $(g_{t})_{t>0}$ such that the quantity $\lambda_{1}(g_{t}) \mbox{Vol}(M,g_{t})^{2/n}$ diverges to infinity as $t$ goes to infinity \cite{CD}. (See also \cite{Bleecker}, \cite{MU}, \cite{Muto}, \cite{Tanno} and \cite{Urakawa}.) However, for a given Riemannian metric $g$ on $M$, the restriction of the functional $\lambda_{1}$ to metrics in the conformal class with fixed volume is bounded \cite{LY, EIconf}. El Soufi-Ilias \cite{EI1} proved that a metric $g$ is extremal for the functional $\lambda_{1}$ among such metrics if and only if there exists a finite collection of eigenfunctions $\{f_{j}\}_{j=1}^{N}$ such that $F:= (f_{1}, \cdots, f_{N}): (M,g) \rightarrow \mathbf{R}^{N}$ is a harmonic map into a unit sphere in $\mathbf{R}^{N}$ with constant energy density $|dF|^{2} =\lambda_{1}(g)$.

Bourguignon--Li--Yau \cite{BLY} proved the following result:
\begin{theo}[\cite{BLY}]
\label{BLY-theo}
Let $(M,J)$ be a compact complex $n$-dimensional manifold admitting a full holomorphic immersion $\Phi: (M, J) \rightarrow \mathbf{C}P^{N}$. Let $\sigma_{FS}$ be the Fubini-Study form on $\mathbf{C}P^{N}$ with constant holomorphic sectional curvature $1$. Then, for any K\"{a}hler form $\omega$ on $(M, J)$, the first eigenvalue $\lambda_{1}(\omega)$ satisfies
 \begin{equation*}
 \lambda_{1}(\omega) \leq n\frac{N+1}{N} \frac{\int_{M}\Phi^{*}\sigma_{FS} \wedge \omega^{n-1}}{\int_{M} \omega^{n}}.
 \end{equation*}
 \end{theo} 
The divergence theorem implies that $\lambda_{1}(\omega)$ is bounded by a constant depending on only $n$, $N$, $\Phi$ and the K\"{a}hler class $[\omega]$. The above theorem implies that the Fubini-Study metric on $\mathbf{C}P^{N}$ is a $\lambda_{1}$-maximizer in its K\"{a}hler class. Biliotti-Ghigi \cite{BG} generalized the result and showed that the canonical K\"{a}hler-Einstein metric on a Hermitian symmetric space of compact type is a $\lambda_{1}$-maximizer in its K\"{a}hler class. (See also \cite{AGL}.) Motivated by these results, Apostolov--Jakobson--Kokarev \cite{AJK} proved that the metric $g$ on a compact K\"{a}hler manifold is extremal for the functional $\lambda_{1}$ within its K\"{a}hler class if and only if there exists a finite collection of eigenfunctions $\{f_{j}\}_{j=1}^{N}$ such that the equation
\begin{equation}
\label{AJKintro}
\lambda_{1}(g)^{2}\left( \sum_{j=1}^{N} f_{j}^{2} \right) -2\lambda_{1}(g) \left(\sum_{j=1}^{N} |\nabla f_{j}|^{2} \right) + \sum_{j=1}^{N} |dd^{c}f_{j}|^{2} = 0
\end{equation}
holds. Using this equation, they showed that the metric $g$ of a compact homogeneous K\"{a}hler-Einstein manifold $(M, J, g, \omega)$ of positive scalar curvature is extremal for the functional $\lambda_{1}$ within its K\"{a}hler class. However, compared to the aforementioned theorems due to Nadirashvili and El Soufi-Ilias, the geometric meaning of the equation $(\ref{AJKintro})$ is not clear.

Let $(M,J)$ be a compact complex manifold satisfying the assumption of Theorem \ref{BLY-theo}. Let $H^{1,1}(M,J;\mathbf{R}) := H^{1,1}(M,J) \cap H^{2}_{dR}(M)$. Then the map 
\begin{equation*}
H^{1,1}(M,J;\mathbf{R}) \rightarrow \mathbf{R}, \quad [\omega] \mapsto \int_{M} \Phi^{*}\sigma_{FS} \wedge \omega^{n-1}
\end{equation*}
is a well-defined continuous function. Thus this is bounded on the compact subset $\{ [\omega] \in H^{1,1}(M,J;\mathbf{R})  \mid \int_{M} \omega^{n}=1 \}$. In other words, the functional $\lambda_{1}$ is bounded on the set of K\"{a}hler metrics with fixed volume on $(M, J)$. However, the property of the $\lambda_{1}$-maximizing K\"{a}hler metrics has not been studied.   

In this paper, on a compact complex manifold $(M, J)$, we introduce the notion of $\lambda_{k}$-extremal K\"{a}hler metric by considering all volume-preserving deformations of the K\"{a}hler metric. Be cautioned that we fix the complex structure $J$ and consider only $J$-compatible K\"{a}hler metrics. (See Definition \ref{definition} for the precise definition of the $\lambda_{k}$-extremality.) The notion of $\lambda_{k}$-extremality in this paper is stronger than that in the above theorem due to Apostolov--Jakobson--Kokarev.  We prove that the K\"{a}hler metric $g$ is $\lambda_{1}$-extremal if and only if there exists a finite collection of eigenfunctions $\{f_{j}\}_{j=1}^{N}$ such that the equations
\begin{equation}
\label{intro}
\left\{ \,
\begin{aligned}
& H \left( \sum_{j=1}^{N} f_{j}dd^{c}f_{j} \right) = -\omega , \\
& \lambda_{1}(g)^{2}\left( \sum_{j=1}^{N} f_{j}^{2} \right) -2\lambda_{1}(g) \left(\sum_{j=1}^{N} |\nabla f_{j}|^{2} \right) + \sum_{j=1}^{N} |dd^{c}f_{j}|^{2} = 0  \\
\end{aligned}
\right.
\end{equation}
hold (Theorem \ref{conclusion}). Here $H$ is the harmonic projector, which is defined due to the Hodge decomposition on a compact K\"{a}hler manifold. It is obvious that the equation $(\ref{intro})$ implies $(\ref{AJKintro})$. In addition, we can obtain a result that is similar to the aforementioned results due to El Soufi-Ilias \cite{EI2} and Apostolov et al \cite{AJK}. That is, the metric on a compact isotropy irreducible homogeneous K\"{a}hler manifold is $\lambda_{1}$-extremal in our sense (Proposition \ref{irreducible}). We also give an example of a K\"{a}hler metric that is $\lambda_{1}$-extremal within its K\"{a}hler class, but not so for all volume-preserving deformations of the K\"{a}hler metric (Example \ref{illustrate}). 

In the final section of this paper, we consider flat complex tori. We prove that the metric on a flat complex torus is $\lambda_{1}$-extremal within its K\"{a}hler class. Montiel--Ros \cite{MR} showed that among all the real $2$-dimensional tori, the only square torus $\mathbf{R}^{2}/\mathbf{Z}^{2}$ admits an isometric minimal immersion into a $3$-dimensional Euclidean sphere by the first eigenfunctions. Later, using Theorem \ref{Nad-intro}, El Soufi--Ilias \cite{EI2} improved the result. That is, they proved that a real $2$-dimensional torus admits isometric minimal immersion into a Euclidean sphere of some dimension by the first eigenfunctions if and only if the torus is the square torus or the equilateral torus. Very recently, L\"{u}--Wang--Xie \cite{LWX} classified all the $3$-dimensional flat tori and $4$-dimensional flat tori that admit an isometric minimal immersion into a Euclidean sphere of some dimension by the first eigenfunctions. Furthermore, they constructed new examples of flat tori that admit an isometric minimal immersion into a Euclidean sphere by the first eigenfunctions. In Theorem \ref{characterization}, we give a necessary and sufficient condition for the metric on a flat complex torus to be $\lambda_{1}$-extremal in our sense. The notion of $\lambda_{1}$-extremality introduced in this paper is stronger than that in Theorem \ref{Nad-intro}. Hence Theorem \ref{characterization} gives a necessary condition for a flat complex torus to admit an isometric minimal immersion into a Euclidean sphere by the first eigenfunctions. As far as the author knows, this is the only currently known necessary condition for a flat complex torus to admit an isometric minimal immersion into a Euclidean sphere by the first eigenfunctions. We also show that if the multiplicity of the first eigenvalue of a flat complex torus is $2$, then the flat metric on the torus is not $\lambda_{1}$-extremal in our sense (Corollary \ref{zeros}). In addition, we show that an $m$-dimensional flat torus $\mathbf{R}^{m}/D_{m}$, where $D_{m}$ is an $m$-dimensional checkerboard lattice, admits an isometric minimal immersion into a Euclidean sphere by the first eigenfunctions (Proposition \ref{check}). This fact has been already proved by L\"{u}--Wang--Xie \cite{LWX} for $m=3,4$.

\section{A $\lambda_{k}$-Extremal K\"{a}hler Metric}
 Let $(M, J, g)$ be a compact K\"{a}hler manifold (without boundary) of complex dimension $n$. By scaling the metric, we may assume that the volume Vol($M$, $g$) is $1$. Let $\omega$ be the K\"{a}hler form and $d\mu$ the volume form. It is well known that $d\mu$ is given by $d\mu = \omega^{n}/n!$. Below we 
 follow the conventions in \cite{AJK}. We define the exterior differential $d$ by $d = \partial + \overline{\partial}$ and twisted exterior differential $d^{c}$ by $d^{c} = i(\overline{\partial}-\partial)$. Both $d$ and $d^{c}$ are real operators. We clearly have $dd^{c} = 2i\partial \overline{\partial}$. The K\"{a}hler metric $g$ induces a pointwise hermitian inner product on $(1,1)$-forms, which we also denote by $g$. Note that $g$ is symmetric for a pair of real $(1,1)$-forms. We have
 \begin{equation}
 \label{omega-norm}
 |\omega|^{2} = g(\omega, \omega) = n.
 \end{equation}
 Let $\delta$ and $\delta^{c}$ be the $L^{2}(g)$-adjoints of $d$ and $d^{c}$ respectively. We define the Laplacian $\Delta_{g}$ by $\Delta_{g} =d\delta + \delta d$. Eigenvalues of the Laplacian $\Delta_{g}$ acting on functions are nonnegative and we denote them by $0 < \lambda_{1}(g) \leq \lambda_{2}(g) \leq \cdots \lambda_{k}(g) \leq \cdots$. For any $k \in \mathbf{N}$, let $E_{k}(g)$ be the vector space of real-valued eigenfunctions of $\Delta_{g}$ corresponding to $\lambda_{k}(g)$. That is, $E_{k}(g)$ is given by $E_{k}(g) = \mbox{Ker}(\Delta_{g}-\lambda_{k}(g)I)$, where $I$ is the identity map acting on functions. We have 
 \begin{equation}
 \label{Laplacian}
 \Delta_{g} \phi = -2 g^{j\overline{k}} \frac{\partial^{2}\phi}{\partial z^{j} \partial{\overline{z}^{k}}} = -2g(\omega, i\partial \overline{\partial}\phi) = -g(\omega, dd^{c}\phi)
 \end{equation}
 for a function $\phi$. We also have 
 \begin{equation*}
 *\omega = \frac{1}{(n-1)!}\omega^{n-1},
 \end {equation*}
 where $*$ is the Hodge star operator. This implies that the equation
 \begin{equation}
 \label{star}
 \alpha \wedge \omega^{n-1} = (n-1)!\alpha \wedge * \omega = (n-1)! g(\alpha, \omega)\frac{\omega^{n}}{n!} = \frac{1}{n} g(\alpha, \omega) \omega^{n}
 \end{equation}
 holds for a $(1,1)$-form $\alpha$. In particular, the equations (\ref{Laplacian}) and (\ref{star}) imply 
 \begin{equation}
 \label{Laplacian-star}
 ndd^{c}\phi \wedge \omega^{n-1} = g(dd^{c}\phi, \omega)\omega^{n} = -(\Delta_{g} \phi) \omega^{n}.
 \end{equation}
 
 Let $Z^{1,1}(M;\mathbf{R})$ be the real vector space of $d$-closed real $(1,1)$-forms on $M$. Let $Z^{1,1}_{0}(M;\mathbf{R})$ be the subspace defined by
 \begin{equation*}
Z^{1,1}_{0}(M;\mathbf{R}) := \left\{ \alpha \in Z^{1,1}(M;\mathbf{R}) \mid \int_{M} g(\alpha, \omega) d\mu = 0 \right\}.
\end{equation*}
Fix an arbitrary element $\alpha \in Z^{1,1}_{0}(M;\mathbf{R})$. The $(1,1)$-form
\begin{equation}
\label{deformation}
\widetilde{\omega}_{t} := \omega + t\alpha 
\end{equation}
is a K\"{a}hler form for a sufficiently small $t$. In particular, if we consider $\alpha = dd^{c}\psi$ for a real-valued function $\psi$, then $\alpha$ satisfies 
\begin{equation*}
\int_{M} g(\alpha, \omega) d\mu = -\int_{M} \Delta_{g} \psi d\mu = 0,
\end{equation*}
and so we have $\alpha \in Z^{1,1}_{0}(M;\mathbf{R})$. The $1$-parameter family $\widetilde{\omega}_{t} = \omega + tdd^{c}\psi$ is a deformation of $\omega$ in its K\"{a}hler class $[\omega]$, which was studied by Apostolov--Jakobson--Kokarev \cite{AJK}. Let $\widetilde{g}_{t}$ be the K\"{a}hler metric corresponding to the K\"{a}hler form $\widetilde{\omega}_{t}$ in (\ref{deformation}). Set
\begin{equation}
\label{metrics}
g_{t} := \mbox{Vol}(M, \tilde{g}_{t})^{-1/n}\widetilde{g_{t}}, \quad \omega_{t} := \mbox{Vol}(M, \tilde{g}_{t})^{-1/n}\widetilde{\omega_{t}}.
\end{equation}
Then we have $g_{0} = g$ and $\omega_{0} = \omega$. We also see that  $(g_{t})_{t}$ is a volume-preserving $1$-parameter family of K\"{a}hler metrics that depends analytically on $t$, and  $\omega_{t}$ is the K\"{a}hler form associated with $g_{t}$.  Moreover, we can verify that $\left. \frac{d}{dt} \right|_{t=0} \omega_{t} = \alpha$. (See (\ref{omega}) below.)
\begin{defi}
\label{definition}
The K\"{a}hler metric $g$ on a compact K\"{a}hler manifold $(M, J, g, \omega)$ is called \textit{$\lambda_{k}$-extremal (for all the volume-preserving deformations of the K\"{a}hler metric)} if the inequality
\begin{equation*}
\left( \left.\frac{d}{dt}\right|_{t=0^{-}} \lambda_{k}(g_{t}) \right)\cdot \left( \left.\frac{d}{dt}\right|_{t=0^{+}} \lambda_{k}(g_{t}) \right) \leq 0
\end{equation*}
holds for any 1-parameter family of volume-preserving K\"{a}hler metrics $(g_{t})_{t}$ that depends real analytically on $t$.
\end{defi}

\begin{rema}
When we consider whether a K\"{a}hler metric $g$ on $(M,J)$ is $\lambda_{k}$-extremal, we may rescale the metric so that $\mbox{Vol}(M,g) =1$. Let $(g_{t})_{t}$ be a 1-parameter family of volume-preserving K\"{a}hler metrics that depends real analytically on $t$. Let $\omega_{t}$ be the K\"{a}hler form associated with $g_{t}$. El Soufi-Ilias \cite{EI} showed that $\left.\frac{d}{dt}\right|_{t=0^{-}} \lambda_{k}(g_{t}) $ and $\left.\frac{d}{dt}\right|_{t=0^{+}} \lambda_{k}(g_{t})$ depend on only $\omega$ and $\left.\frac{d}{dt}\right|_{t=0}\omega_{t}$. Since $(\omega_{t})_{t}$ is volume-preserving, we have $\left.\frac{d}{dt}\right|_{t=0}\omega_{t} \in Z^{1,1}_{0}(M; \mathbf{R})$. Hence it suffices to consider $(\omega_{t})_{t}$ given by $(\ref{metrics})$. Thus a K\"{a}hler metric $g$ on $(M,J)$ with $\mbox{Vol}(M,g) =1$ is $\lambda_{k}$-extremal if and only if for any $\alpha \in  Z^{1,1}_{0}(M;\mathbf{R})$, the associated volume-preserving $1$-parameter family of K\"{a}hler metrics $(g_{t})_{t}$ defined by (\ref{metrics}) satisfies
\begin{equation*}
\left( \left.\frac{d}{dt}\right|_{t=0^{-}} \lambda_{k}(g_{t}) \right)\cdot \left( \left.\frac{d}{dt}\right|_{t=0^{+}} \lambda_{k}(g_{t}) \right) \leq 0.
\end{equation*}
\end{rema}

We quote the following theorem due to El Soufi-Ilias \cite{EI}:
\begin{theo}[\cite{EI}]
\label{EI}
Let $(M,g)$ be a compact Riemannian manifold and $(g_{t})_{t}$ be a $1$-parameter family of Riemannian metrics that depends real-analytically on $t$ with $g_{0}=g$. Let $\Pi_{k}: L^{2}(M,g) \rightarrow E_{k}(g)$ be the orthogonal projection onto $E_{k}(g)$. Define the operator $P_{k}: E_{k}(g) \rightarrow E_{k}(g)$ by 
\begin{equation}
\label{Pk}
P_{k}(f) := \Pi_{k} \left( \left.\frac{d}{dt}\right|_{t=0} \Delta_{g_{t}}f \right).
\end{equation}
Then the following hold:
\begin{enumerate}
\item $\left. \frac{d}{dt}\right|_{t=0^{-}} \lambda_{k}(g_{t})$ and $\left.\frac{d}{dt}\right|_{t=0^{+}} \lambda_{k}(g_{t})$ are eigenvalues of $P_{k}$. \\
\item If $\lambda_{k}(g) > \lambda_{k-1}(g)$, then $\left. \frac{d}{dt}\right|_{t=0^{-}} \lambda_{k}(g_{t})$ and $\left.\frac{d}{dt}\right|_{t=0^{+}} \lambda_{k}(g_{t})$ are the greatest and the least eigenvalues of $P_{k}$. \\
\item If $\lambda_{k}(g) < \lambda_{k+1}(g)$, then $\left. \frac{d}{dt}\right|_{t=0^{-}} \lambda_{k}(g_{t})$ and $\left.\frac{d}{dt}\right|_{t=0^{+}} \lambda_{k}(g_{t})$ are the least and the greatest eigenvalues of $P_{k}$. \\
\end{enumerate}
\end{theo} 

For any $\alpha \in  Z^{1,1}_{0}(M;\mathbf{R})$, the associated volume-preserving $1$-parameter family of K\"{a}hler metrics $(g_{t})_{t}$ given by (\ref{metrics}) defines the associated operator $P_{k, \alpha}: E_{k}(g) \rightarrow E_{k}(g)$ by (\ref{Pk}). Since $P_{k, \alpha}$ is symmetric with respect to the $L^{2}(g)$-inner product induced by $g$, one can consider the corresponding quadratic form on $E_{k}(g)$, given by
\begin{equation}
\label{quadratic form}
Q_{\alpha}(f) := \int_{M} f P_{k, \alpha}(f) d\mu = \int_{M} f \left( \left.\frac{d}{dt}\right|_{t=0} \Delta_{g_{t}}f \right) d\mu.
\end{equation}
The following proposition is an immediate consequence of Theorem \ref{EI} (1):
\begin{prop} 
\label{extremal} 
Let $(M, J, g, \omega)$ be a compact K\"{a}hler manifold. If the metric $g$ of $(M, J, g, \omega)$ is $\lambda_{k}$-extremal, then the quadratic form $Q_{\alpha}$, defined in (\ref{quadratic form}), is indefinite on $E_{k}(g)$ for any $\alpha \in  Z^{1,1}_{0}(M;\mathbf{R})$.
\end{prop}
We also have the following proposition, which follows form Theorem \ref{EI} (2) and (3):
\begin{prop}
Let $(M, J, g, \omega)$ be a compact K\"{a}hler manifold. Suppose that $\lambda_{k}(g) > \lambda_{k-1}(g)$ or $\lambda_{k}(g) < \lambda_{k+1}(g)$ holds. Then the metric $g$ of $(M, J, g)$ is $\lambda_{k}$-extremal if and only if the quadratic form $Q_{\alpha}$ is indefinite on $E_{k}(g)$ for any $\alpha \in  Z^{1,1}_{0}(M;\mathbf{R})$. 
\end{prop}

The next corollary immediately follows.

\begin{corr}
\label{converse}
Let $(M, J, g, \omega)$ be a compact K\"{a}hler manifold. Then the metric $g$ is $\lambda_{1}$-extremal if and only if the quadratic form $Q_{\alpha}$ is indefinite on $E_{1}(g)$ for any $\alpha \in  Z^{1,1}_{0}(M;\mathbf{R})$.
\end{corr}

\begin{theo}
Let $(M, J, g, \omega)$ be a compact K\"{a}hler manifold of complex dimension $n$ with $\mbox{Vol}(M,g) =1$. For any $\alpha \in  Z^{1,1}_{0}(M;\mathbf{R})$, the quadratic form $Q_{\alpha}$, given by (\ref{quadratic form}), can be expressed as 
\begin{equation*}
Q_{\alpha}(f) = \int_{M}g(fdd^{c}f, \alpha) d\mu.
\end{equation*}
\end{theo}

\begin{proof}
First we calculate $\left.\frac{d}{dt}\right|_{t=0}\mbox{Vol}(M, \widetilde{g_{t}})$ and  $\left.\frac{d}{dt}\right|_{t=0} \omega_{t}$. The volume form $\widetilde{d\mu}_{t}$, determined by $\widetilde{\omega}_{t}$, can be written as 
\begin{equation*}
\widetilde{d\mu}_{t} = \frac{1}{n!} \widetilde{\omega}_{t}^{n} = \frac{1}{n!}\left(\omega^{n} + tn\alpha \wedge \omega^{n-1} \right) + O(t^{2}) = [1+tg(\alpha, \omega)] d\mu + O(t^{2}).
\end{equation*}
Hence one obtains
\begin{equation*}
\mbox{Vol}(M, \widetilde{g}_{t}) = \int_{M}  \widetilde{d\mu}_{t} = 1 + t\int_{M}g(\alpha, \omega) d\mu +O(t^{2}) = 1+O(t^{2}),
\end{equation*}
where in the last equality we have used the the assumptions that $\mbox{Vol}(M, \widetilde{g}) = 1$ and $\alpha \in  Z^{1,1}_{0}(M;\mathbf{R})$. Hence this implies $\left.\frac{d}{dt}\right|_{t=0} \mbox{Vol}(M, \tilde{g_{t}}) = 0$. Thus one obtains
\begin{equation}
\label{omega}
\left. \frac{d}{dt} \right|_{t=0} \omega_{t} = \frac{d}{dt} \left( \mbox{Vol}(M, \widetilde{g}_{t})^{-1/n})  \tilde{\omega}_{t} \right) = \mbox{Vol}(M, \tilde{g}_{0})^{-1/n} \left. \frac{d}{dt} \right|_{t=0} \widetilde{\omega}_{t} = \alpha.
\end{equation}

Next we differentiate 
\begin{equation*}
ndd^{c}f \wedge \omega_{t}^{n-1} = -(\Delta_{g_{t}} f) \omega_{t}^{n},
 \end{equation*}
which  comes from (\ref{Laplacian-star}). Differentiating the left hand side at $t=0$, one obtains 
\begin{equation*}
n(n-1)dd^{c}f \wedge \left( \left.\frac{d}{dt}\right|_{t=0} \omega_{t}\right) \wedge \omega^{n-2} = n (n-1) dd^{c}f \wedge \alpha \wedge \omega^{n-2},
\end{equation*}
where (\ref{omega}) is used. On the other hand, differentiating the right hand side at $t=0$, one obtains
\begin{equation*}
\begin{split}
&\quad -\left( \left.\frac{d}{dt}\right|_{t=0} \Delta_{g_{t}}f \right)\omega^{n} -n(\Delta_{g}f)\left( \left.\frac{d}{dt}\right|_{t=0} \omega_{t} \right) \wedge \omega^{n-1} \\
&= -\left( \left.\frac{d}{dt}\right|_{t=0} \Delta_{g_{t}}f \right)\omega^{n} - n(\Delta_{g}f)\alpha \wedge \omega^{n-1} \\
&= -\left( \left.\frac{d}{dt}\right|_{t=0} \Delta_{g_{t}}f \right)\omega^{n} - (\Delta_{g}f) g(\alpha, \omega) \omega^{n}, \\
\end{split}
\end{equation*}
where (\ref{star}) is used for the last equality. Thus one obtains 
\begin{equation}
\label{derivative}
\left( \left.\frac{d}{dt}\right|_{t=0} \Delta_{g_{t}}f \right)\omega^{n} = -n(n-1)dd^{c}f \wedge \alpha \wedge \omega^{n-2} - g(\alpha, \omega)\omega^{n}.
\end{equation}
It is known \cite{AJK} that the equation
\begin{equation*}
g(\phi, \psi) \omega^{n} = g(\phi, \omega) g(\psi, \omega)\omega^{n} - n(n-1)\phi \wedge \psi \wedge \omega^{n-2}
\end{equation*} 
holds for a pair of real $(1,1)$-forms $\phi$ and $\psi$. Substituting $\phi = dd^{c}f$ and $\psi = \alpha$ into this equation and combining with (\ref{derivative}) and (\ref{Laplacian}), one obtains
\begin{equation*}
\left( \left.\frac{d}{dt}\right|_{t=0} \Delta_{g_{t}}f \right)\omega^{n} = g(dd^{c}f, \alpha)\omega^{n}.
\end{equation*}
Hence one concludes 
\begin{equation*}
\left.\frac{d}{dt}\right|_{t=0} \Delta_{g_{t}}f  = g(dd^{c}f, \alpha)
\end{equation*}
and thus the assertion follows.
\end{proof}

For a compact K\"{a}hler manifold $(M, J, g, \omega)$, $\mathcal{H}^{1,1}(M)$, the vector space of $\mathbf{C}$-valued harmonic $(1,1)$-forms, is known to be finite dimensional. In particular, $\mathcal{H}^{1,1}(M)$ is a closed subspace of $\Omega^{1,1}(M)$ and hence we can define the $\mathbf{C}$-linear $L^{2}(g)$-orthogonal projection $H : \Omega^{1,1}(M) \rightarrow \mathcal{H}^{1,1}(M)$. For a real $(1,1)$-form $\eta$, $H(\eta)$ is a real harmonic $(1,1)$-form. Let $\mathcal{H}^{1,1}(M; \mathbf{R})$ be the vector space of real harmonic $(1,1)$-forms. Set 
\begin{equation*}
\mathcal{H}^{1,1}_{0}(M; \mathbf{R}) :=  \left\{ \alpha \in \mathcal{H}^{1,1}(M;\mathbf{R}) \mid \int_{M} g(\alpha, \omega) d\mu = 0 \right\}
\end{equation*}
and 
\begin{equation*}
C^{\infty}_{0}(M; \mathbf{R}) := \left\{ \varphi \in C^{\infty}(M;\textbf{R}) \mid \int_{M} \varphi d\mu = 0 \right\}.
\end{equation*}
By the $dd^{c}$-lemma, $\alpha \in Z^{1,1}_{0}(M;\mathbf{R})$ can be decomposed as 
\begin{equation*}
\alpha = H(\alpha) + dd^{c} \varphi.
\end{equation*}
This decomposition gives an $\mathbf{R}$-linear bijection between $Z^{1,1}_{0}(M;\mathbf{R})$ and $\mathcal{H}^{1,1}_{0}(M; \mathbf{R}) \times C^{\infty}_{0}(M; \mathbf{R})$. 

Apostolov-Jakobson-Kokarev\cite{AJK} introduced the fourth order differential operator $L(f) := \delta^{c}\delta (fdd^{c}f)$, where $\delta$ and $\delta^{c}$ are the $L^{2}$-adjoints of $d$ and $d^{c}$ respectively. They proved that the equation 
\begin{equation*}
L(f) = \lambda_{k}(g) f^{2} -2\lambda_{k}(g) |\nabla f|^{2} + |dd^{c} f|^{2} 
\end{equation*}
holds for an eigenfunction $f \in E_{k}(g)$.

We have the following theorem:
\begin{theo}
\label{harmonic}
Let $(M, J, g, \omega)$ be a compact K\"{a}hler manifold of complex dimension $n$. The following are equivalent:
\begin{enumerate}
\item For any $\alpha \in  Z^{1,1}_{0}(M;\mathbf{R})$, the quadratic form $Q_{\alpha}$ given by (\ref{quadratic form}) is indefinite on the eigenspace $E_{k}(g)$. \\
\item There exists a finite subset $\{ f_{1}, \cdots , f_{N} \} \subset E_{k}(g)$ such that the following equations hold:
\begin{equation}
\label{harmonic-L}
\left\{ \,
\begin{aligned}
& H \left( \sum_{j=1}^{N} f_{j}dd^{c}f_{j} \right) = -\omega \\
&\sum_{j=1}^{N}L(f_{j} ) = \lambda_{k}(g)^{2}\left( \sum_{j=1}^{N} f_{j}^{2} \right) -2\lambda_{k}(g) \left(\sum_{j=1}^{N} |\nabla f_{j}|^{2} \right) + \sum_{j=1}^{N} |dd^{c}f_{j}|^{2} = 0  \\
\end{aligned}
\right.
\end{equation}
\\
\end{enumerate}
\end{theo}
The proof of this theorem is inspired by that of Lemma 3.1 in \cite{EI}.

\begin{proof}
We may assume that $\mbox{Vol}(M,g) =1$. We assume the condition $(1)$. Let $K$ be the convex hull of $\{ \left( H(fdd^{c}f) , L(f) \right) \mid f \in E_{k}(g) \}$ in $\mathcal{H}^{1,1}(M; \mathbf{R}) \times C^{\infty}(M; \mathbf{R})$. Since $E_{k}(g)$ is finite dimensional, $K$ is contained in a finite dimensional subspace of $\mathcal{H}^{1,1}(M; \mathbf{R}) \times C^{\infty}(M; \mathbf{R})$. We consider the product $L^{2}$-inner metric on the subspace. We assume that $(-\omega, 0) \notin K$. Then the hyperplane separation theorem implies that there exists $\left(\widetilde{\alpha}_{H}, \widetilde{\varphi} \right) \in \mathcal{H}^{1,1}(M; \mathbf{R}) \times C^{\infty}(M; \mathbf{R})$ such that the inequalities 
\begin{equation}
\label{inequalities}
\int_{M} g(-\omega, \widetilde{\alpha}_{H}) d\mu < 0 \quad \mbox{and} \quad \int_{M} g(\eta, \widetilde{\alpha}_{H}) d\mu + \int_{M}s\varphi d\mu \geq 0
\end{equation}
hold for all $(\eta, s) \in K\setminus{\{0\}}$. Consider $\alpha_{H} \in  \mathcal{H}^{1,1}_{0}(M;\mathbf{R})$ and $\varphi \in C^{\infty}_{0}(M; \mathbf{R})$ respectively defined by 
\begin{equation}
\label{alpha}
\alpha_{H} := \widetilde{\alpha}_{H}- \frac{1}{n} \left( \int_{M} g(\omega, \widetilde{\alpha}_{H}) d\mu \right) \omega.
\end{equation}
\begin{equation}
\label{phi}
\varphi := \widetilde{\varphi}- \int_{M} \widetilde{\varphi} d\mu.
\end{equation}
Set 
\begin{equation*}
\alpha := \alpha_{H} + dd^{c}\varphi = \alpha_{H} + dd^{c}\widetilde{\varphi}.
\end{equation*}
Then for any $f \in E_{k}(g)\setminus{\{0\}}$, one has 
\begin{equation*}
\begin{split}
Q_{\alpha}(f) &= \int_{M} g(fdd^{c}f, \alpha) d\mu \\
&= \int_{M} g(fdd^{c}f, \alpha_{H}) d\mu + \int_{M}g(fdd^{c}f, dd^{c}\widetilde{\varphi}) d\mu \\
&= \int_{M} g(H(fdd^{c}f), \widetilde{\alpha}_{H}) d\mu - \frac{1}{n} \left[ \int_{M}g(fdd^{c}f, \omega) d\mu \right] \left[\int_{M} g(\omega, \widetilde{\alpha}_{H} ) d\mu  \right] + \int_{M}L(f)\widetilde{\varphi} d\mu \\
& = \int_{M} g(H(fdd^{c}f), \widetilde{\alpha}_{H}) d\mu + \int_{M}L(f)\widetilde{\varphi} d\mu + \frac{\lambda_{k}(g) }{n} \left[ \int_{M} f^{2} d\mu \right] \left[\int_{M}g(\omega, \widetilde{\alpha}_{H}) d\mu \right] \\
&>0.
\end{split}
\end{equation*}
This contradicts the condition (1) and hence one concludes that $(-\omega, 0) \in K$. This implies that $(1) \Rightarrow (2)$.

Conversely, we assume that there exists a finite subset $\{ f_{1}, \cdots , f_{N} \} \subset E_{k}(g)$ satisfying $(\ref{harmonic-L})$. Take any $\alpha_{H} \in \mathcal{H}^{1,1}_{0}(M; \mathbf{R})$ and $\varphi \in C^{\infty}_{0}(M; \mathbf{R})$ and consider $\alpha \in Z^{1,1}_{0}(M; \mathbf{R})$ defined by $\alpha := \alpha_{H} + dd^{c} \varphi$.
Then one has 
\begin{equation*}
\begin{split}
\sum_{j=1}^{N}Q_{\alpha}(f_{j}) &= \sum_{j=1}^{N} \int_{M} g (f_{j}dd^{c}f_{j}, \alpha) d\mu \\
&= - \int_{M} g\left(\sum_{j=1}^{N}f_{j}dd^{c}f_{j},  \alpha_{H}  \right) d\mu + \sum_{j=1}^{N} \int_{M}g(f_{j}dd^{c}f_{j}, dd^{c}\varphi) d\mu \\
&= \int_{M}g(-\omega, \alpha_{H}) d\mu + \sum_{j=1}^{N}\int_{M} L(f_{j}) \varphi d\mu \\
&=0. \\
\end{split}
\end{equation*}
This implies that $Q_{\alpha}$ is indefinite on $E_{k}(g)$. This completes the proof.
\end{proof}



Combining Corollary \ref{converse} and Theorem \ref{harmonic}, we conclude the following:

\begin{theo}
\label{conclusion}
Let $(M, J, g, \omega)$ be a compact K\"{a}hler manifold. Suppose that the K\"{a}hler metric $g$ is $\lambda_{k}$-extremal. Then there exists a finite subset $\{ f_{1}, \cdots , f_{N} \} \subset E_{k}(g)$ satisfying $(\ref{harmonic-L})$. For $k=1$, the existence of such a finite collection of eigenfunctions is a necessary and sufficient condition for the K\"{a}hler metric $g$ to be $\lambda_{1}$-extremal.
\end{theo}

We remark the relationship between Theorem \ref{conclusion} and Theorem 2.1 in \cite{AJK}, which we quote as follows:
\begin{theo}[\cite{AJK}]
\label{AJKtheo}
Let $(M, J, g)$ be a compact K\"{a}hler manifold. The K\"{a}hler metric $g$ is $\lambda_{k}$-extremal for all the deformations of the K\"{a}hler metric within its K\"{a}hler class. Then there exists a finite subset $\{ f_{1}, \cdots , f_{N} \} \subset E_{k}(g)$ such that the equation 
\begin{equation}
\label{AJKeq}
\sum_{j=1}^{N}L(f_{j}) = \lambda_{k}(g)^{2}\left( \sum_{j=1}^{N} f_{j}^{2} \right) -2\lambda_{k}(g) \left(\sum_{j=1}^{N} |\nabla f_{j}|^{2} \right) + \sum_{j=1}^{N} |dd^{c}f_{j}|^{2} = 0
\end{equation}
holds. For $k=1$, the existence of such a finite collection of eigenfunctions is a necessary and sufficient condition for the K\"{a}hler metric $g$ to be $\lambda_{1}$-extremal for all the deformations of the K\"{a}hler metric within its K\"{a}hler class. 
\end{theo}

Obviously, the condition $(\ref{harmonic-L})$ is stronger than the condition $(\ref{AJKeq})$. This fact is natural since our $\lambda_{k}$-extremality is stronger than that in \cite{AJK}. Using Theorem \ref{AJKtheo} above, one can show that if the K\"{a}hler metric $g$ is extremal for the functional $\lambda_{k}$ within its K\"{a}hler class, then the eigenvalue $\lambda_{k}(g)$ is multiple \cite{AJK}. Hence we obtain the same conclusion for our $\lambda_{k}$-extremality:

\begin{corr}
Let $(M, J, g, \omega)$ be a compact K\"{a}hler manifold. Suppose that $g$ is $\lambda_{k}$-extremal for all the volume-preserving deformations of the K\"{a}hler metric. Then the eigenvalue $\lambda_{k}(g)$ is multiple.
\end{corr}

El Soufi-Ilias\cite{EI2} showed that the metric on a compact isotropy irreducible homogeneous space is $\lambda_{1}$-extremal with respect to all the volume-preserving deformations of the Riemannian metric. Since their $\lambda_{1}$-exremality is stronger than ours, we have the following proposition:
\begin{prop}
\label{irreducible}
Let $G/K$ be a compact isotropy irreducible homogeneous K\"{a}hler manifold. Then the metric is $\lambda_{1}$-extremal for all the volume-preserving deformations of the K\"{a}hler metric.
\end{prop}
We remark that one can also prove this proposition directly from Theorem \ref{conclusion}, using a similar discussion to that in \cite[Section 3]{Takahashi}. It is known that the metric of a compact isotropy irreducible homogeneous K\"{a}hler manifold is Einstein\cite{Wolf}. An irreducible Hermitian symmetric space of compact type is a compact isotropy irreducible homogeneous K\"{a}hler manifold. In fact, the converse also holds. That is, a compact isotropy irreducible homogeneous K\"{a}hler manifold is an irreducible Hermitian symmetric space of compact type \cite{WZ, Lichnerowicz}. Apotolov--Jakobson--Kokarev\cite{AJK} proved that the metric on a compact homogeneous K\"{a}her-Einstein manifold of positive scalar curvature is $\lambda_{1}$-extremal within its K\"{a}hler class. 

Before stating the corollaries of Theorem \ref{conclusion}, we recall basic facts about the first eigenvalue of the Laplacian on a product of Riemannian manifolds. Let $(M, g)$ and $(M', g')$ be Riemannian manifolds. For a function $f \in C^{\infty}(M; \mathbf{R})$, we define the function $f \times 1$ on $M \times M'$ by
\begin{equation*}
(f \times 1)(x, y) := f(x), \quad (x,y) \in M \times M'.
\end{equation*}
For a function $h \in C^{\infty}(M'; \mathbf{R})$, we define the function $1 \times h$ on $M \times M'$ in a similar manner. Suppose that $\lambda_{1}(g) \leq \lambda_{1}(g')$. Then the first eigenvalue $\lambda_{1}(g\times g')$ of the product Riemannian manifold $(M \times M', g\times g')$ is equal to $\lambda_{1}(g)$. $E_{1}(g \times g')$, the space of $\lambda_{1}(g \times g')$-eigenfunctions on $M \times M'$ is given by
\begin{equation*}
E_{1}(g \times g') =
\begin{cases}
\mbox{span}\{ f \times 1, 1 \times h \mid f \in E_{1}(g), h \in E_{1}(g') \} & (\mbox{if } \lambda_{1}(g) = \lambda_{1}(g') )\\
\mbox{span}\{ f \times 1 \mid f \in E_{1}(g) \} & (\mbox{if } \lambda_{1}(g) < \lambda_{1}(g') ).\\
\end{cases}
\end{equation*}

The following corollary follows from Theorem \ref{conclusion}: 
\begin{corr}
\label{product}
Let $(M, J, g, \omega)$ and $(M', J' , g', \omega')$ be compact K\"{a}hler manifolds. Assume that $\lambda_{1}(g) = \lambda_{1}(g')$ and that $g$ and $g'$ are both $\lambda_{1}$-extremal for all the volume-preserving deformations of the K\"{a}hler metrics. Then the product K\"{a}hler metric $g \times g'$ on $(M, J) \times (M, J)$ is $\lambda_{1}$-extremal for all the volume-preserving deformations of the K\"{a}hler metric. 
\end{corr}   

\begin{proof}
By hypothesis, there exist finite subsets $\{f_{j}\} \subset E_{1}(g)$ and $\{h_{k} \} \subset E_{1}(g')$ such that
\begin{equation*}
\left\{ \,
\begin{aligned}
&\sum_{j}H_{M} \left( f_{j}dd^{c}f_{j} \right) = -\omega, \quad \sum_{j}L_{M}(f_{j}) = 0 \\
&\sum_{k}H_{M'} \left(  h_{k}dd^{c}h_{k} \right) = -\omega', \quad \sum_{k}L_{M'}(h_{k}) = 0. \\
\end{aligned}
\right.
\end{equation*}
Then one has 
\begin{equation*}
\begin{split}
& \quad \sum_{j}L_{M\times M'} (f_{j} \times 1) + \sum_{k} L_{M\times M'} (1 \times h_{k}) \\
&= \sum_{j} L_{M}(f_{j}dd^{c}f_{j}) + \sum_{k} L_{M'}(h_{k}) \\
&=0.
\end{split}
\end{equation*}
We also have
\begin{equation*}
\begin{split}
&\quad \sum_{j}H_{M\times M'} \left( (f_{j}\times 1)d_{M\times M'}d^{c}_{M\times M'}(f_{j}\times 1) \right) \\
&= \sum_{j} H_{M\times M'} (f_{j}d_{M}d^{c}_{M}f_{j} \oplus 0) \\
&= \sum_{j} H_{M}(f_{j}d_{M}d^{c}_{M}f_{j}) \oplus 0 \\
& = -\omega \oplus 0. \\
\end{split}
\end{equation*}
Similarly, we obtain
\begin{equation*}
\sum_{k} H_{M\times M'} \left( (1\times h_{k} )d_{M\times M'}d^{c}_{M\times M'}(1 \times h_{k} ) \right) = 0 \oplus -\omega'.
\end{equation*}
Hence we obtain
\begin{equation*}
\sum_{j} H\left( (f_{j}\times 1)dd^{c}(f_{j}\times 1) \right) + \sum_{k} H \left( (1\times h_{k} )dd^{c}(1 \times h_{k} ) \right) = -\omega \oplus -\omega', 
\end{equation*}
where we omit the subscript $M \times M'$. Thus one concludes that $\{ f_{j} \times 1 \}_{j} \cup \{1 \times h_{k} \}_{k}$ satisfy the equation $(\ref{harmonic-L})$. The proof is completed.
\end{proof}
From the above proof, one can immediately obtain the following corollary:
 
 \begin{corr}
 \label{NOT}
 Let $(M, J, g, \omega)$ and $(M', J' , g', \omega')$ be compact K\"{a}hler manifolds. Assume that $\lambda_{1}(g) \neq \lambda_{1}(g')$. Then the product K\"{a}hler metric $g \times g'$ on $(M, J) \times (M', J')$ is not $\lambda_{1}$-extremal for all the volume-preserving deformations of the K\"{a}hler metric. 
 \end{corr}

The above corollaries also correspond to the Corollary 2.3 in \cite{AJK}, which we quote in the following:
\begin{prop}[\cite{AJK}]
\label{AJKproduct}
Let $(M, J, g, \omega)$ and $(M', J' , g', \omega')$ be compact K\"{a}hler manifolds. Suppose that $g$ is $\lambda_{1}$-extremal for all the deformations of the K\"{a}hler metric in its K\"{a}hler class and that the inequality $\lambda_{1}(g) \leq \lambda_{1}(g')$ holds. Then the product metric $g \times g'$ is $\lambda_{1}$-extremal for all the deformations of the K\"{a}hler metric in its K\"{a}hler class on $(M, J) \times (M', J')$.
\end{prop}
We can see that $\lambda_{1}$-extremality in Corollary \ref{product} is stronger than that in Proposition \ref{AJKproduct}. To clarify the difference, let us look at the following example:
\begin{exam}
\label{sym}
Let $(M, J, g, \omega)$, $(M', J', g', \omega')$ be irreducible Hermitian symmetric spaces of compact type with $\rho = c \: \omega$ and $\rho ' = c'\omega '$ for some $c, c'>0$, where $\rho$ and $\rho '$ are the Ricci forms on $M$ and $M'$ respectively. By Example \ref{irreducible}, $g$ is $\lambda_{1}$-extremal for all the volume-preserving deformations of the K\"{a}hler metric and so is $g'$. The result due to Nagano \cite{Nagano} shows that $\lambda_{1}(g) = 2c$ and $\lambda_{1}(g') =2c'$ (see also \cite{TK}). By Proposition \ref{AJKproduct}, the product K\"{a}hler metric $g\times g'$ on $(M, J) \times (M', J')$ is $\lambda_{1}$-extremal within its K\"{a}hler class. However, Corollary \ref{product} and \ref{NOT} imply that the metric $g\times g'$ is $\lambda_{1}$-extremal for all the volume-preserving deformations of the K\"{a}hler metric if and only if $c=c'$. 
\end{exam}
The simplest case of this example is the following:
\begin{exam}
\label{illustrate}
Let $g_{FS}$ be the Fubini-Study metric on the complex projective space $\mathbf{C}P^{n}$. Take any $c>0$. Then the product K\"{a}hler metric $g_{FS} \times c g_{FS}$ on $\mathbf{C}P^{n} \times \mathbf{C}P^{n}$ is $\lambda_{1}$-extremal in its K\"{a}hler class. However, $g_{FS} \times c g_{FS}$ is $\lambda_{1}$-extremal for all the volume-preserving deformations of the K\"{a}hler metric if and only if $c=1$.
\end{exam}


El Soufi-Ilias \cite{EI1} showed that if a metric $g$ of a Riemannian manifold $(M,g)$ is $\lambda_{k}$-extremal for volume-preserving conformal deformations, then there exists a finite subset $\{f_{1}, \cdots , f_{N} \} \subset E_{k}(g)$ such that $F :=(f_{1}, \cdots f_{N}): M \rightarrow \mathbf{R}^{N}$ is a harmonic map to a sphere with constant energy density. For a Riemann surface $(M, J, g, \omega)$, conformal deformations and deformations within the K\"{a}hler class are equivalent, so it is natural to expect the following result due to Apostolov-Jakobson-Kokarev \cite{AJK}:
\begin{prop}[\cite{AJK}]
Let $(M, J, g, \omega)$ be a compact Riemann surface without boundary. Let $\{f_{1}, \cdots , f_{N} \} \subset E_{k}(g)$ be eigenfunctions associated with $\lambda_{k}(g)$. The following two conditions are equivalent:
\begin{enumerate}
\item $F :=(f_{1}, \cdots f_{N}): M \rightarrow \mathbf{R}^{N}$ gives a harmonic map to $S^{N-1}(c)$ with $|dF|^{2} = c\lambda_{k}(g)$ for some $c>0$, where $S^{N-1}(c)$ is the unit sphere of radius $c$. \\
\item $\sum_{j=1}^{N} L(f_{j}) = 0.$ \\
\end{enumerate}
\end{prop}
To prove the result, they used the identity $dd^{c}\psi = - (\Delta_{g} \psi) \omega$ and showed that the equation
\begin{equation*}
\sum_{j} L(f_{j}) =2\lambda_{k}(g) \left(\lambda_{k}(g)\sum_{j}f_{j}^{2}-\sum_{j}|\nabla f_{j}|^{2}\right)= \lambda_{k}(g) \Delta_{g}\left(\sum_{j}f_{j}^{2}\right)
\end{equation*}
holds for any finite subset $\{f_{j} \} \subset E_{k}(g)$. Since we have
\begin{equation*}
H\left( \sum_{j}f_{j}dd^{c}f_{j} \right) = \lambda_{k}(g)H\left( (\sum_{j} f_{j}^{2}) \omega \right),
\end{equation*}
one can see that if $\{f_{j}\} \subset E_{k}(g)$ satisfies $\sum_{j}L(f_{j})=0$, then the equation $H(\sum_{j}f_{j}dd^{c}f_{j}) = -\omega$ holds after an appropriate rescaling of $\{f_{j}\}$. Hence Theorem \ref{conclusion} and Theorem \ref{AJKtheo} are equivalent for a Riemann surface. This is also a natural result since deformations within the K\"{a}hler class and general K\"{a}hler deformations are equivalent for a compact Riemann surface.

\section{Complex Tori}
In view of Theorem $\ref{conclusion}$, the harmonic projector $H$ and information of the space of eigenfunctions are important. However, it is hard to find an explicit formula for the harmonic projector $H$ on a general K\"{a}hler manifold. It is also hard to determine the space of eigenfunctions explicitly on a general Riemannian manifold. However, $H$ and the eigenfunctions can be written explicitly for a complex torus. Since it is hard to compute $(\ref{harmonic-L})$ for arbitrary collection of eigenfunctions, we compute it for the $L^{2}$-orthonormal basis of the space of eigenfunctions. From the calculations, we see that the metric on any flat complex torus is $\lambda_{1}$-extremal within its K\"{a}hler class. In addition, we deduce a condition for the flat metric to be $\lambda_{1}$-extremal for all the volume-preserving deformations of the K\"{a}hler metric. 

Let $\gamma_{1}, \cdots, \gamma_{2n}$ be vectors in $\mathbf{C}^{n}$ that are linearly independent over $\mathbf{R}$. We denote by $\Gamma$ the lattice in $\mathbf{C}^{n}$ with basis $\{ \gamma_{1}, \cdots, \gamma_{2n} \}$. Let $z_{1}, \cdots, z_{n}$ be the standard complex coordinates of $\mathbf{C}^{n}$ and $x^{1}, \cdots , x^{2n}$ the real coordinates defined by 
\begin{equation*}
z^{k} = x^{2k-1} + ix^{2k} \quad (k=1, \cdots, n).
\end{equation*}
The lattice $\Gamma$ naturally acts on $\mathbf{C}^{n}$ by translation. Then the quotient space $T_{\Gamma}^{n} := \mathbf{C}^{n}/\Gamma$ becomes a complex manifold in a natural way. $T_{\Gamma}^{n}$ is called a \textit{complex torus}. The standard metric on $\mathbf{C}^{n}$ is given by $\sum_{j=1}^{n} dz^{j} \otimes d\overline{z}^{j}$. This defines a K\"{a}hler form 
\begin{equation*}
\widetilde{\omega} := \frac{i}{2} \sum_{j=1}^{n} dz^{j} \wedge d\overline{z}^{j}
\end{equation*}
on $\mathbf{C}^{n}$. The canonical holomorphic projection $\mathbf{C}^{n} \rightarrow \mathbf{C}^{n}/\Gamma = T_{\Gamma}^{n}$ induces a K\"{a}hler metric $g$ and a K\"{a}hler form $\omega$ on $T_{\Gamma}^{n}$. If we express $w_{1}$, $w_{2} \in \mathbf{C}^{n}$ as  
\begin{equation*}
w_{k} = (w_{k}^{1}, \cdots ,w_{k}^{n}), \quad w_{k} = (u_{k}^{1}, \cdots , u_{k}^{2n}), \quad w_{k}^{j} = u_{k}^{2j-1}+iu_{k}^{2j} \quad (k=1, 2)
\end{equation*}
in the complex coordinates $(z^{1}, \cdots z^{n})$ and the real coordinates $(x^{1}, \cdots, x^{2n})$ respectively, then the standard inner product of $w_{1}$ and $w_{2}$ is given by
\begin{equation*}
\sum_{j=1}^{2n}u_{1}^{j}u_{2}^{j} = \frac{1}{2}(w_{1}\cdot \overline{w}_{2} + \overline{w}_{1}\cdot w_{2}).
\end{equation*}
We define the dual lattice of $\Gamma$, which is denoted by $\Gamma^{*}$, by
\begin{equation*}
\begin{split}
\Gamma^{*} &= \{ w \in \mathbf{C}^{n}  \mid \frac{1}{2}(v \cdot \overline{w} + \overline{v} \cdot w) \in \mathbf{Z} \quad \mbox{for all $v \in \Gamma$} \} \\
&= \{ w \in \mathbf{C}^{n}  \mid \mbox{exp}\left( \pi i (v \cdot \overline{w} + \overline{v} \cdot w) \right) = 1 \quad \mbox{for all  $v \in \Gamma$} \}. \\
\end{split}
\end{equation*}
For any $w \in \Gamma^{*}$, we define a function $\Phi_{w} : T_{\Gamma}^{n} \rightarrow \mathbf{C}$ by 
\begin{equation*}
\Phi_{w}(z) : = \mbox{exp}\left( \pi i (z \cdot \overline{w} + \overline{z} \cdot w) \right). 
\end{equation*}
Then $\Phi_{w}$ is actually a well-defined function on the complex torus $T_{\Gamma}^{n}$. It is known that $\lambda$ is an eigenvalue of the Laplacian on $(T_{\Gamma}^{n}, g)$ if and only if there exists $w \in \Gamma^{*} $ such that $\lambda = 4\pi^{2}|w|^{2}$, where $|w|^{2} = w \cdot \overline{w}$. We set
\begin{equation*}
S(\lambda) :=  \{ w \in \Gamma^{*} \mid \lambda = 4\pi^{2}|w|^{2} \}.
 \end{equation*}
 Then the multiplicity of $\lambda$ is given by $\# S(\lambda)$. For $\lambda \neq 0$, the number $\# S(\lambda)$ is an even integer and $S(\lambda)$ can be written as 
 \begin{equation*}
 S(\lambda) = \{ \pm w_{1}, \pm w_{2}, \cdots , \pm w_{l(\lambda)} \},
 \end{equation*}
 where each $w_{\nu}$ $(\nu =1, \cdots ,l(\lambda))$ is an element of $\Gamma^{*}$ with $\lambda = 4\pi^{2}|w_{\nu}|^{2}$. Set
 \begin{equation*}
 \varphi_{w}(z) := \sqrt{\frac{2}{\mbox{Vol}(T^{n}_{\Gamma})}} \mbox{Re}\left(\Phi_{w}(z)\right) = \sqrt{\frac{2}{\mbox{Vol}(T^{n}_{\Gamma})}} \cos\left(2\pi \sum_{k=1}^{2n}x^{k}u^{k} \right)  ,
 \end{equation*}
 \begin{equation*}
  \psi_{w}(z) := \sqrt{\frac{2}{\mbox{Vol}(T^{n}_{\Gamma})}} \mbox{Im}\left(\Phi_{w}(z)\right) = \sqrt{\frac{2}{\mbox{Vol}(T^{n}_{\Gamma})}} \sin\left(2\pi \sum_{k=1}^{2n}x^{k}u^{k}\right),
 \end{equation*} 
 where we use the real coordinates $z= (x^{1}, \cdots, x^{2n})$ and $w= (u^{1}, \cdots, u^{2n})$.
One can verify that $\{ \varphi_{w_{\nu}} , \psi_{w_{\nu}} \mid \nu =1, \cdots ,l(\lambda) \}$ is an $L^{2}$-orthonormal basis of the real eigenspace $E(\lambda)$. (For details of the aforementioned spectral property of flat tori,  see \cite[pp.272-273]{Sakai}, for instance.)

Apostolov-Jakobson-Kokarev\cite{AJK} proved that the metric on a compact homogeneous K\"{a}hler-Einstein manifold of positive scalar curvature is $\lambda_{1}$-extremal within its K\"{a}hler class. We show that the metric on a flat complex torus is also $\lambda_{1}$-extremal within its K\"{a}hler class.
\begin{prop}
\label{AJK-torus}
Let $(T^{n}_{\Gamma}, g)$ be a flat complex torus. Then the metric $g$ is $\lambda_{1}$-extremal within its K\"{a}hler class.
\end{prop}
\begin{proof}
In the proof, we use the notations introduced above. By Theorem $\ref{AJKtheo}$, it suffices to show that the $L^{2}$-orthonormal basis $\{ \varphi_{w_{\nu}} , \psi_{w_{\nu}} \mid \nu =1, \cdots ,l(\lambda) \}$ of $E_{1}(g)$ satisfy
\begin{equation}
\label{torus-class}
\sum_{\nu=1}^{l(\lambda_{1}(g) )} L(\varphi_{w_{\nu}}) + L(\psi_{w_{\nu}}) = 0.
\end{equation}
By a straightforward calculation, we have 
\begin{equation*}
|\nabla \varphi_{w_{\nu}} |^{2} = \frac{2}{\mbox{Vol}(T^{n}_{\Gamma}) } \left[ 4\pi^{2}\sum_{j=1}^{2n}(u_{\nu}^{j})^{2} \right] \sin^{2}\left(2\pi \sum_{k=1}^{2n}x^{k}u_{\nu}^{k}\right) = \lambda_{1}(g)\psi_{w}^{2}.
\end{equation*} 
Similarly, we have
\begin{equation*}
|\nabla \psi_{w_{\nu}} |^{2} = \lambda_{1}(g)\varphi_{w}^{2}.
\end{equation*}
On the other hand, it is easy to obtain
\begin{equation}
\label{ddc1}
dd^{c}\varphi_{w_{\nu}} = -2\pi^{2}i \varphi_{w_{\nu}} \sum_{\alpha, \beta=1}^{n} \overline{w}_{\nu}^{\alpha} w_{\nu}^{\beta} dz^{\alpha} \wedge d\overline{z}^{\beta}. 
\end{equation}
Hence we obtain
\begin{equation*}
\begin{split}
|dd^{c}\varphi_{w_{\nu}}|^{2} &= 4\pi^{4}\varphi_{w_{\nu}}^{2} \sum_{\alpha, \beta, \gamma, \zeta} g^{\alpha \overline{\zeta}} g^{\overline{\beta}\gamma} \overline{w}_{\nu}^{\alpha} w_{\nu}^{\beta} \overline{w}_{\nu}^{\gamma} w_{\nu}^{\zeta} \\
&= 16\pi^{4} \varphi_{w_{\nu}}^{2} \sum_{\alpha, \beta} |w_{\nu}^{\alpha}|^{2}|w_{\nu}^{\beta}|^{2} \\
&=\lambda_{1}(g)^{2}\varphi_{w_{\nu}}^{2}. \\
\end{split}
\end{equation*}
Similarly, we have 
\begin{equation}
\label{ddc2}
dd^{c}\psi_{w_{\nu}} = -2\pi^{2}i \psi_{w_{\nu}} \sum_{\alpha, \beta=1}^{n} \overline{w}_{\nu}^{\alpha} w_{\nu}^{\beta} dz^{\alpha} \wedge d\overline{z}^{\beta}. 
\end{equation}
and
\begin{equation*}
|dd^{c}\psi_{w_{\nu}}|^{2}  =\lambda_{1}(g)^{2}\psi_{w_{\nu}}^{2}.
\end{equation*}
Thus for each $\nu$, we have
\begin{equation}
\label{lplusl}
\begin{split}
&L(\varphi_{w_{\nu}}) + L(\psi_{w_{\nu}}) \\
&= \left(\lambda_{1}(g)^{2}\varphi_{w_{\nu}} - 2\lambda_{1}(g)^{2}\psi_{w_{\nu}}^{2} + \lambda_{1}(g)^{2}\varphi_{w_{\nu}}^{2} \right) + \left(\lambda_{1}(g)^{2}\psi_{w_{\nu}} - 2\lambda_{1}(g)^{2}\varphi_{w_{\nu}}^{2} + \lambda_{1}(g)^{2}\psi_{w_{\nu}}^{2} \right) \\
&=0. \\
\end{split}
\end{equation}
Thus $(\ref{torus-class})$ is proved.
\end{proof} 

The harmonic projector $H: \Omega^{1,1}(T_{\Gamma}^{n}) \rightarrow \mathcal{H}^{1,1}(T_{\Gamma}^{n})$ is given by
\begin{equation*}
H(\phi) =\frac{1}{\mbox{Vol}(T_{\Gamma}^{n})} \sum_{\alpha,\beta=1}^{n} \left(\int_{T_{\Gamma}^{n}} \phi_{\alpha \overline{\beta}} d\mu \right) dz^{\alpha} \wedge d\overline{z}^{\beta}
\end{equation*}
for a $(1,1)$-form $\phi = \sum_{\alpha, \beta=1}^{n}  \phi_{\alpha\overline{\beta}} dz^{\alpha} \wedge d\overline{z}^{\beta}$, where $d\mu$ is the volume form of $(T_{\Gamma}^{n}, g)$. (For details of the Hodge decomposition on a complex torus, see \cite[Section 1.4]{BL}.) Using this fact and Theorem \ref{conclusion}, we prove the following:
\begin{theo}
\label{characterization}
Let $(T_{\Gamma}^{n},g)$ be a flat $n$-dimensional complex torus. Let $\{w_{\nu} \}_{\nu=1}^{l(\lambda_{k}(g))}$ be linearly independent vectors in $\Gamma^{*}$ satisfying $\lambda_{k} (g)= 4\pi^{2}|w_{\nu}|^{2}$. If the flat metric $g$ is $\lambda_{k}$-extremal for all the volume-preserving deformations of the K\"{a}hler metric, then there exists $\{R_{\nu}  \geq 0\}_{\nu=1}^{l(\lambda_{k}(g))}$ such that the following equations hold:
\begin{equation}
\label{R}
\left\{ \,
\begin{aligned}
&\sum_{\nu=1}^{l(\lambda_{k}(g))} R_{\nu} \overline{w}_{\nu}^{\alpha}w_{\nu}^{\beta} = 0 \quad \mbox{for} \quad 1\leq  \alpha \neq \beta \leq n, \\
&\sum_{\nu=1}^{l(\lambda_{k}(g))} R_{\nu} |w_{\nu}^{\alpha}|^{2} = 1  \quad \mbox{for} \quad 1 \leq \alpha  \leq n.\\
\end{aligned}
\right.
\end{equation}
For $k=1$, the existence of such $\{R_{\nu}  \geq 0\}_{\nu=1}^{l(\lambda_{1}(g))}$ is also a sufficient condition for the metric $g$ to be $\lambda_{1}$-extremal for all the volume-preserving deformations of the K\"{a}hler metric.
\end{theo}

\begin{proof}
First we prove the fist half of the assertion. By theorem \ref{conclusion}, we must have 
\begin{equation}
\label{assumption}
H \left( \sum_{j=1}^{N} f_{j}dd^{c}f_{j} \right) = -\frac{i}{2} \sum_{\alpha=1}^{n} dz^{\alpha} \wedge d\overline{z}^{\alpha}
\end{equation}
for some finite collection of eigenfunctions $\{f_{j} \}_{j=1}^{N} \subset E_{k}(g)$. Each eigenfunction $f_{j}$ is of the form 
\begin{equation*}
f_{j}(z) = \sum_{\nu=1}^{l(\lambda_{k}(g))} a_{j\nu} \varphi_{w_{\nu}} + b_{j\nu} \psi_{w_{\nu}} \quad (a_{j\nu}, b_{j\nu} \in \mathbf{R}).
\end{equation*}
Using $(\ref{ddc1})$ and $(\ref{ddc2})$, we obtain
\begin{equation*}
\begin{split}
 f_{j}dd^{c}f_{j} &=  -2\pi^{2}i \sum_{\alpha, \beta=1}^{n} \sum_{\nu, \tau=1}^{l(\lambda_{k}(g))} \overline{w}_{\nu}^{\alpha}w_{\nu}^{\beta} [ a_{j\nu}a_{j\tau} \varphi_{w_{\nu}} \varphi_{w_{\tau}} + a_{j\nu}b_{j\tau} \varphi_{w_{\nu}} \psi_{w_{\tau}} \\
&\quad + a_{j\tau}b_{j\nu} \varphi_{w_{\tau}} \psi_{w_{\nu}} + b_{j\nu} b_{j\tau} \psi_{w_{\nu}} \psi_{w_{\tau}} ] dz^{\alpha} \wedge d\overline{z}^{\beta}. \\
\end{split}
\end{equation*}
Hence one obtains
\begin{equation*}
 H(f_{j}dd^{c}f_{j}) =  \frac{-2\pi^{2}i}{\mbox{Vol}(T_{\Gamma}^{n})}  \sum_{\alpha,\beta=1}^{n}  \sum_{\nu=1}^{l(\lambda_{k}(g))} \overline{w}_{\nu}^{\alpha}w_{\nu}^{\beta}(  a_{j\nu}^{2} + b_{j\nu}^{2} ) dz^{\alpha} \wedge d\overline{z}^{\beta}. 
\end{equation*}
Thus one obtains 
\begin{equation}
\label{tori-harmonic}
 \sum_{j=1}^{N} H (f_{j}dd^{c}f_{j}) = \frac{-2\pi^{2}i}{\mbox{Vol}(T_{\Gamma}^{n})}  \sum_{\alpha,\beta=1}^{n}  \sum_{\nu=1}^{l(\lambda_{k}(g))} \overline{w}_{\nu}^{\alpha}w_{\nu}^{\beta}(\sum_{j=1}^{N}  a_{j\nu}^{2} + b_{j\nu}^{2} ) dz^{\alpha} \wedge d\overline{z}^{\beta}. 
\end{equation}
Then the equations (\ref{assumption}) and (\ref{tori-harmonic}) imply that we must have
\begin{equation*}
\left\{ \,
\begin{aligned}
&\sum_{\nu=1}^{l(\lambda_{k}(g))} \overline{w}_{\nu}^{\alpha}w_{\nu}^{\beta}(\sum_{j=1}^{N}  a_{j\nu}^{2} + b_{j\nu}^{2} ) = 0 \quad \mbox{for} \quad 1\leq  \alpha \neq \beta \leq n, \\
&\sum_{\nu=1}^{l(\lambda_{k}(g))} |\overline{w}_{\nu}^{\alpha}|^{2}(\sum_{j=1}^{N}  a_{j\nu}^{2} + b_{j\nu}^{2} ) = \frac{ \mbox{Vol}(T_{\Gamma}^{n}) }{4\pi^{2}}  \quad \mbox{for} \quad 1 \leq \alpha \leq n.\\
\end{aligned}
\right.
\end{equation*}
Setting $R_{\nu} := 4\pi^{2}(\sum_{j=1}^{N}  a_{j\nu}^{2} + b_{j\nu}^{2} )/\mbox{Vol}(T_{\Gamma}^{n}) $, one concludes the first half of the assertion.

Next we prove the second half. We assume the existence of $\{R_{\nu}  \geq 0\}_{\nu=1}^{l(\lambda_{1}(g))}$ satisfying $(\ref{R})$. We use Theorem \ref{conclusion} to prove the proposition. $\{ \varphi_{w_{\nu}} , \psi_{w_{\nu}} \mid \nu =1, \cdots ,l(\lambda_{1}(g)) \}$ is an $L^{2}$-orthonormal basis of $E(\lambda_{1}(g))$. By $(\ref{lplusl})$, we immediately have
\begin{equation*}
L(\sqrt{R_{\nu}} \varphi_{w_{\nu}} ) + L(\sqrt{R_{\nu} }\psi_{w_{\nu}}) = 0
\end{equation*}
for each $\nu$. Hence we have
\begin{equation*}
 \sum_{\nu=1}^{l(\lambda_{1}(g)} L(\sqrt{R_{\nu}} \varphi_{w_{\nu}} ) + L(\sqrt{R_{\nu} }\psi_{w_{\nu}}) = 0.
\end{equation*}
Thus it suffices to prove
\begin{equation*}
\sum_{\nu=1}^{l(\lambda_{1}(g))} R_{\nu} \left[H(\varphi_{w_{\nu}}dd^{c}\varphi_{w_{\nu}}) + H(\psi_{w_{\nu}}dd^{c}\psi_{w_{\nu}}) \right] = -a\omega
\end{equation*}
for some $a>0$. Using the equations $(\ref{ddc1})$ and $(\ref{ddc2})$, one obtains
\begin{equation*}
\begin{split}
&\quad \sum_{\nu=1}^{l(\lambda_{1}(g))} R_{\nu}\left[ H(\varphi_{w_{\nu}}dd^{c}\varphi_{w_{\nu}}) + H(\psi_{w_{\nu}}dd^{c}\psi_{w_{\nu}}) \right] \\
&=  \frac{-2\pi^{2}i}{\mbox{Vol}(T_{\Gamma}^{n})} \sum_{\nu=1}^{l(\lambda_{1}(g))} \sum_{\alpha,\beta=1}^{n} R_{\nu}   \left(\int_{T_{\Gamma}^{n}} (\phi_{w_{\nu}}^{2}+\psi_{w_{\nu}}^{2})d\mu \right) \overline{w}_{\nu}^{\alpha} w_{\nu}^{\beta} dz^{\alpha} \wedge d\overline{z}^{\beta}\\
&=  \frac{-4\pi^{2}i}{\mbox{Vol}(T_{\Gamma}^{n})} \sum_{\nu=1}^{l(\lambda_{1}(g))}  \sum_{\alpha,\beta=1}^{n} R_{\nu} \overline{w}_{\nu}^{\alpha} w_{\nu}^{\beta} dz^{\alpha} \wedge d\overline{z}^{\beta}.\\
\end{split}
\end{equation*}
By hypothesis, the proof is completed.
\end{proof}
The implication of this theorem is not clear, so we consider simple cases in what follows. First we consider the case where $\mbox{dim}E_{k}(g) =2$, that is, $l(\lambda_{k}(g)) =1$. Then we have the following corollary:

\begin{corr}
\label{zeros}
Let $(T_{\Gamma}^{n},g)$ be a flat $n$-dimensional complex torus. Suppose that $\mbox{dim}E_{k}(g) =2$ holds for some $k$. Then the metric $g$ is not $\lambda_{k}$-extremal.
\end{corr}

\begin{proof}
We prove the assertion by contradiction. Assume that the metric $g$ is $\lambda_{k}$-extremal. By hypothesis, there exists a pair $w$, $-w \in \Gamma^{*} \subset \mathbf{C}^{n}$ uniquely up to sign such that  $\lambda_{k}(g) = 4\pi^{2}|w|^{2}$. First we show that the vector $w$ is of the form $w = (0, \cdots , \xi , \cdots , 0)$ for some $\xi \in \mathbf{C}$. The first equation in $(\ref{R})$ implies $\overline{w}^{\alpha}w^{\beta}=0$ for any pair of distinct integers $(\alpha,\beta)$. Since $w = (w^{1}, \cdots , w^{n})$ is a nonzero vector, we have $w^{\alpha} \neq 0$ for some $\alpha$.  Let $w^{j} = u^{2j-1}+iu^{2j}$ for each $1 \leq j \leq n$. Then for any $\beta \neq \alpha$, we have 
\begin{equation}
\label{Left}
u^{2\alpha-1}u^{2\beta-1} = -u^{2\alpha}u^{2\beta} 
\end{equation}
 and
 \begin{equation}
 \label{Right}
 u^{2\alpha}u^{2\beta-1} = u^{2\alpha-1} u^{2\beta}.
\end{equation}
Assume that $u^{2\alpha} \neq 0$ and $u^{2\beta} \neq 0$. Then by $(\ref{Right})$, there exists $c \in \mathbf{R}$ such that $u^{2\alpha-1} = cu^{2\alpha}$ and $u^{2\beta-1} = cu^{2\beta}$. Substituting these for $(\ref{Left})$, one obtains
\begin{equation*}
c^{2}u^{2\alpha}u^{2\beta} = -u^{2\alpha}u^{2\beta}.
\end{equation*}
This is a contradiction and so we have $u^{2\alpha}=0$ or $u^{2\beta}=0$. 

If we have $u^{2\alpha}=0$, then we must have $u^{2\alpha-1} \neq 0$ since we now assume $w^{\alpha} \neq 0$. Hence by $(\ref{Right})$, we have $u^{2\beta} = 0$. Then $(\ref{Left})$ immediately implies $u^{2\beta-1} = 0$. Thus we have $u^{2\beta-1} = u^{2\beta} = 0$, that is $w^{\beta}=0$.

If we have $u^{2\beta}=0$, then $(\ref{Right})$ implies that we have $u^{2\alpha}=0$ or $u^{2\beta-1}=0$. We have already considered the case where $u^{2\alpha}=0$. Hence we consider the case where $u^{2\beta-1}=0$, but this immediately implies $w^{\beta}=0$.

Thus we conclude that $w$ is of the form $w = (0, \cdots , \xi , \cdots , 0)$. However, this contradicts with the second equation in $(\ref{R})$. The proof is completed.
\end{proof}

\begin{exam}
\label{standard}
\underline{The standard lattice: $\Gamma = \mathbf{Z}^{2n}$}. Consider the standard complex torus $\textbf{C}^{n}/\mathbf{Z}^{2n}$ with the flat metric $g$. Let $\{e_{j} \}_{j=1}^{n}$ be the standard orthonormal basis of $\textbf{C}^{n}$. Set
\begin{equation*}
w_{2k-1} := e_{k},  \quad  w_{2k} := ie_{k}
\end{equation*}
for every $1 \leq k \leq n$. Then we have $S(\lambda_{1}(g)) = \{\pm w_{\nu} \}_{\nu=1}^{2n}$ and $l(\lambda_{1}(g)) = 2n$. It is clear that $(\ref{R})$ is equivalent to the condition where $R_{2k-1}+R_{2k} =1$ for any $1\leq k \leq n$ and so the torus $\textbf{C}^{n}/\mathbf{Z}^{2n}$ satisfies the assumption of Proposition \ref{characterization}. Hence the metric $g$ is $\lambda_{1}$-extremal for all the volume-preserving deformations of the K\"{a}hler metric. This fact is not new since the metric is $\lambda_{1}$-extremal for all the volume-preserving metric deformations. This can be seen from Theorem \ref{Nad-intro} and the classical fact that the standard torus admits an isometric minimal immersion into a unit sphere by first eigenfunctions as follows:
\begin{equation*}
\begin{split}
&\mathbf{C}^{n}/\mathbf{Z}^{2n}  \rightarrow S^{4n-1}\left(\sqrt{\frac{n}{2\pi^{2}} }\right), \\
& (x^{1}, \ldots, x^{2n}) \mapsto \left(\frac{1}{2\pi}\cos(2\pi x^{1}), \frac{1}{2\pi}\sin(2\pi x^{1}), \ldots, \frac{1}{2\pi}\cos(2\pi x^{2n}), \frac{1}{2\pi}\sin(2\pi x^{2n}) \right). \\
\end{split}
\end{equation*}
\end{exam}

\begin{exam}
\label{checkerboard4}
\underline{The checkerboard lattice}. First we consider the (real) $4$-dimensional checkerboard lattice $D_{4}$, which is defined by
\begin{equation*}
D_{4} := \{ (x^{1}, \ldots , x^{4}) \in \mathbf{Z}^{4} \mid x^{1} + \cdots +x^{4} \in 2\mathbf{Z} \}.
\end{equation*} 
$D_{4}$ is self-dual, i.e. $D_{4} \cong D^{*}_{4}$. The dual lattice $D_{4}^{*} (\cong D_{4})$ is known to be the lattice in $\mathbf{C}^{2}$ with the basis $(1,0)$, $(0,1)$, $(i,0)$, $\left( \frac{1+i}{2}, \frac{1+i}{2}\right)$. (See \cite[pp.117-120]{CS}, for instance.) Set
\begin{equation*}
\begin{split}
&w_{1} := (1,0), \quad  w_{2}:= (0,1), \quad w_{3} := (i, 0), \quad w_{4} := (0, i), \\
&w_{5} :=\left(\frac{1+i}{2}, \frac{1+i}{2} \right), \quad w_{6}:= \left( \frac{1-i}{2}, \frac{1-i}{2} \right), \\
&w_{7}:= \left( \frac{1+i}{2}, -\frac{1+i}{2} \right), \quad w_{8} := \left(\frac{1-i}{2} , -\frac{1-i}{2} \right), \\
&w_{9} := \left( \frac{1+i}{2},  \frac{1-i}{2} \right), \quad  w_{10} := \left( \frac{1+i}{2}, -  \frac{1-i}{2} \right), \\
& w_{11} := \left(\frac{1-i}{2},  \frac{1+i}{2}\right), \quad w_{12} := \left(\frac{1-i}{2},  -\frac{1+i}{2}\right).\\
\end{split}
\end{equation*}
If we set $R_{1} = \cdots =R_{4} = 1/4$, $R_{5} = \cdots =R_{12} = 1/8$, then it is elementary to check that the equations $(\ref{R})$ hold for $k=1$. Hence by Theorem \ref{characterization}, the flat metric $g$ on the $2$-dimensional complex torus $\mathbf{C}^{2}/D_{4}$ is $\lambda_{1}$-extremal for all the volume-preserving deformations of the K\"{a}hler metric. This fact is not new since the metric is $\lambda_{1}$-extremal for all the volume-preserving metric deformations. This follows from the fact that L\"{u}--Wang--Xie \cite{LWX} recently found a $2$-parameter family of isometric minimal immersion by the first eigenfunctions from $\mathbf{C}^{2}/D_{4}$ into the unit sphere $S^{23} \subset \mathbf{R}^{24}$. (See Example 1.1 in \cite{LWX}.)

In fact, for any $m \geq 3$, the checkerboard lattice $D_{m}$ is defined as a lattice in $\mathbf{R}^{m}$ by
\begin{equation*}
D_{m} := \{ (x^{1}, \ldots , x^{m}) \in \mathbf{Z}^{m} \mid x^{1} + \cdots +x^{m} \in 2\mathbf{Z} \}.
\end{equation*} 
We show the following: 
\begin{prop}
\label{check}
For any $m \geq 3$, the flat torus $\mathbf{R}^{m}/D_{m}$ admits an isometric minimal immersion into a Euclidean sphere. 
\end{prop}
\begin{proof}
The property of $D_{m}$ should be considered separately for the case $m=3$, $m=4$ and $m \geq 5$. For $m=3$ and $m=4$, the assertion has been proved by L\"{u}--Wang--Xie \cite{LWX}. (See Example $4.3$ in \cite{LWX} for $m=3$ and Example 1.1 in \cite{LWX} for $m=4$.) Hence it suffices to consider the case where $m\geq 5$. For $m \geq 5$, $D_{m}^{*}$ is a lattice with the basis $\{e_{j} \}_{j=1}^{m-1} \cup \{ \frac{1}{2}\sum_{k=1}^{m} e_{k} \}$, where $\{ e_{j} \}_{j=1}^{m}$ is the standard basis in $\mathbf{R}^{m}$. (See \cite[p.120]{CS}, for instance.) The shortest vectors are exactly the $2m$ vectors $\{ \pm e_{j} \}_{j=1}^{m}$. Thus we have $\lambda_{1}(g) = 4\pi^{2}$ and $E_{1}(g)$ is spanned by $\{ \cos(2\pi x^{j}), \sin(2\pi x^{j}) \}_{j=1}^{m}$, where $\{x^{j} \}_{j=1}^{m}$ is the standard coordinate in $\mathbf{R}^{m}$. It is obvious that the map
\begin{equation*}
\begin{split}
&\mathbf{R}^{m}/D_{m}  \rightarrow S^{2m-1}\left(\frac{ \sqrt{m} }{2\pi}\right), \\
& (x^{1}, \ldots, x^{m}) \mapsto \left(\frac{1}{2\pi}\cos(2\pi x^{1}), \frac{1}{2\pi}\sin(2\pi x^{1}), \ldots, \frac{1}{2\pi}\cos(2\pi x^{m}), \frac{1}{2\pi}\sin(2\pi x^{m}) \right) \\
\end{split}
\end{equation*}
is an isometric minimal immersion.
\end{proof}
For a long time, only the standard torus (Example \ref{standard}) had been an example of higher dimensional tori that admit an isometric minimal immersion into a Euclidean sphere by the first eigenfunctions. Very recently, L\"{u}--Wang--Xie \cite{LWX} constructed new examples of higher dimensional flat tori that admit an isometric minimal immersion into a Euclidean sphere by the first eigenfunctions. The flat torus $\mathbf{R}^{m}/D_{m}$ $(m\geq 3)$ is a new example.
\end{exam}

\begin{exam}
\label{1a1b}
For $a, b \in [1, \infty)$, consider the lattice $\Gamma_{a,b}$ in $\mathbf{C}^{2}$ with the lattice basis $(1,0)$, $(a^{-1}\sqrt{-1}, 0)$, $(0,1)$, $(0, b^{-1}\sqrt{-1})$. Let $T^{2}_{a,b}$ be the $2$-dimensional complex torus determined by $\Gamma_{a,b}$ with the flat metric $g_{a,b}$. Let $\Gamma_{a} \subset \mathbf{C}$ be the lattice with the lattice basis $(1,0)$, $(a^{-1}\sqrt{-1}, 0)$ and $(T^{1}_{a}, h_{a})$ the flat $1$-dimensional complex torus determined by $\Gamma_{a}$. Then $(T^{2}_{a,b}, g_{a,b})$ is the product of $(T^{1}_{a}, h_{a})$ and $(T^{1}_{b}, h_{b})$. We have $\lambda_{1}(T^{1}_{a}, h_{a}) = 1 = \lambda_{1}(T^{1}_{b}, h_{b})$. Hence Proposition \ref{AJK-torus} and Corollary \ref{product} imply that the metric $g_{a,b}$ on $T^{2}_{a,b}$ is $\lambda_{1}$-extremal for all the volume-preserving deformations of the K\"{a}hler metric. However, since we have $E_{1}(g_{a,b}) = \mbox{span}\{ \cos(2\pi x^{1}), \sin(2\pi x^{1}), \cos(2\pi x^{3}), \sin(2\pi x^{3}) \}$, the flat torus $(T^{2}_{a,b}, g_{a,b})$ does not admit an isometric minimal immersion into a Euclidean sphere by first eigenfunctions. Thus $g_{a,b}$ is not $\lambda_{1}$-extremal for all the volume-preserving metric deformations.
\end{exam}

\begin{exam}
\label{11ab}
For $a, b \in [1, \infty)$, consider the lattice $\widetilde{\Gamma}_{a,b}$ in $\mathbf{C}^{2}$ with the lattice basis $(1,0)$, $(\sqrt{-1}, 0)$, $(0,a^{-1})$, $(0, b^{-1}\sqrt{-1})$. Let $\widetilde{T}^{2}_{a,b}$ be the $2$-dimensional complex torus determined by $\widetilde{\Gamma}_{a,b}$ with the flat metric $\widetilde{g}_{a,b}$. Let $(T^{1}_{std}, h_{std})$ be the flat 1-dimensional complex torus determined by the lattice with the lattice basis $(1,0)$, $(\sqrt{-1}, 0)$. Let $(T^{1}_{a,b}, h_{a,b})$ be the flat 1-dimensional complex torus determined by the lattice with the lattice basis $(0,a^{-1})$, $(0, b^{-1}\sqrt{-1})$. Then $(\widetilde{T}^{2}_{a,b}, \widetilde{g}_{a,b})$ is the product of $(T^{1}_{std}, h_{std})$ and $(T^{1}_{a,b}, h_{a,b})$. We have $\lambda_{1}(T^{1}_{std}, h_{std}) = 1$ and $\lambda_{1}(T^{1}_{a,b}, h_{a,b}) = \mbox{min}\{a,b\}$. Hence if we have $a=1$ or $b=1$, then Proposition \ref{AJK-torus} and Corollary \ref{product} imply that the metric $\widetilde{g}_{a,b}$ on $\widetilde{T}^{2}_{a,b}$ is $\lambda_{1}$-extremal for all the volume-preserving deformations of the K\"{a}hler metric. On the other hand, if we have $a>1$ and $b>1$, then by Corollary \ref{NOT}, $\widetilde{g}_{a,b}$ is not $\lambda_{1}$-extremal for all the volume-preserving deformations of the K\"{a}hler metric. 
\end{exam}
In Example \ref{1a1b} and Example \ref{11ab}, if we ignore the complex structure on $\mathbf{C}^{2}$ and regard $\mathbf{C}^{2}$ as $\mathbf{R}^{4}$, then we have $\Gamma_{a,b} \cong \widetilde{\Gamma}_{a,b}$. However, whether the flat metric is $\lambda_{1}$-extremal is different in Example \ref{1a1b} and Example \ref{11ab}. Hence Example \ref{1a1b} and Example \ref{11ab} show that the notion of $\lambda_{1}$-extremality actually depends on the complex structure.

Finally we give a 1-parameter family of 2-dimensional complex tori whose flat metrics are not $\lambda_{1}$-extremal for all the volume-preserving deformations of the K\"{a}hler metric.
\begin{exam}
\label{last}
For $\pi/3 < \theta <\pi /2$, we consider the lattice $\Gamma_{\theta} \subset \mathbf{C}^{2}$ with the lattice basis $(1, 0)$, $(\cos \theta, \sin \theta)$, $(\sqrt{-1}, 0)$, $(\sqrt{-1}\cos \theta, \sqrt{-1}\sin \theta)$. Let $g_{\theta}$ be the flat metric on $\mathbf{C}^{2}/\Gamma_{\theta}$. It is straightforward to check that the dual lattice $\Gamma_{\theta}^{*}$ is the lattice with the basis $w_{1} := (1, -\cos \theta /\sin \theta)$, $w_{2}:= (0, 1/\sin \theta)$, $w_{3}:=(\sqrt{-1}, -\cos \theta /\sin \theta)$, $w_{4}:=(0, \sqrt{-1}/\sin \theta)$. Then we have $S(\lambda_{1}(g_{\theta})) = \{ \pm w_{\nu} \}_{\nu=1}^{4}$ and so the multiplicity of $\lambda_{1}(g_{\theta})$ is $8$. If $g$ is $\lambda_{1}$-extremal, then Theorem \ref{characterization} implies that there exists $\{ R_{\nu} \}_{\nu=1}^{4}$ such that
\begin{equation*}
-\frac{\cos \theta}{\sin \theta}(R_{1}+R_{3}) = 0, \quad R_{1}+R_{3}= 1, \quad \frac{1}{\sin^{2}\theta}(R_{2}+R_{4}) =1.
\end{equation*}
Since we have $\pi/3 < \theta <\pi /2$, this is a contradiction. Hence $g_{\theta}$ is not $\lambda_{1}$-extremal. $\mathbf{C}^{2}/\Gamma_{\theta}$ is not a product of $1$-dimensional flat complex tori. In fact, assume that $\mathbf{C}^{2}/\Gamma_{\theta}$ is a product of $(T_{1},h_{1})$ and $(T_{2}, h_{2})$, where each is a $1$-dimensional flat complex torus. If we had $\lambda_{1}(h_{1}) = \lambda_{1}(h_{2})$, then by Proposition \ref{AJK-torus} and Corollary \ref{product}, $g_{\theta}$ would be $\lambda_{1}$-extremal. Hence we have $\lambda_{1}(h_{1}) \neq \lambda_{1}(h_{2})$. We may assume $\lambda_{1}(h_{1}) < \lambda_{1}(h_{2})$. Then the multiplicity of of $\lambda_{1}(h_{1})$ is equal to that of $\lambda_{1}(g_{\theta})$, that is, $8$. This is a contradiction since the multiplicity of the first eigenvalue of a $1$-dimensional flat complex torus is at most $6$ (see \cite{EI2}, for example). Thus $\mathbf{C}^{2}/\Gamma_{\theta}$ is not a product of $1$-dimensional complex tori. Let $\widetilde{\Gamma}_{\theta} \subset \mathbf{R}^{2}$ be the lattice with the lattice basis $(1,0)$, $(\cos \theta, \sin \theta)$. Let $(\mathbf{R}^{2}/\widetilde{\Gamma}_{\theta}, h_{ \theta})$ be the flat real 2-dimensional torus. If we ignore the complex structure on $\mathbf{C}^{2}$ and regard it as $\mathbf{R}^{4}$, then $(\mathbf{R}^{4}/\Gamma_{\theta}, g_{\theta})$ is a Riemannian product of $(\mathbf{R}^{2}/\widetilde{\Gamma}_{\theta}, h_{ \theta})$ and $(\mathbf{R}^{2}/\widetilde{\Gamma}_{\theta}, h_{ \theta})$.
\end{exam} 

\textbf{Acknowledgements} I would like to thank my advisor, Professor Shin Nayatani, for his constant encouragement and valuable suggestions. I am also grateful to Professor Fabio Podest\`{a} for his interest in this research.

\section*{Declarations}

\begin{itemize}
\item Funding This work is partially supported by the Grant-in-Aid for JSPS Fellows Grant Number JP23KJ1074.
\item Competing interests The author declares no competing interests.
\item Ethics approval and consent to participate `Not applicable'
\item Consent for publication `Not applicable'
\item Data availability `Not applicable'
\item Materials availability `Not applicable'
\item Code availability `Not applicable'
\item Author contribution The author declares sole responsibility for this study.
\end{itemize}

\end{document}